\documentclass[11pt]{article}
\usepackage{amscd, amsmath, amssymb, amsthm}
\usepackage[all,cmtip]{xy}
\usepackage[pagebackref]{hyperref}

\title{Hodge structures of type $(n,0,\ldots,0,n)$}
\author{Burt Totaro}
\date{  }

\def\Z{\text{\bf Z}}
\def\Q{\text{\bf Q}}
\def\R{\text{\bf R}}
\def\C{\text{\bf C}}
\def\P{\text{\bf P}}

\def\HH{\text{\bf H}}

\def\arrow{\rightarrow}
\def\inj{\hookrightarrow}

\def\Gal{\text{Gal}}
\def\End{\text{End}}
\def\Gr{\text{Gr}}
\def\isot{\text{isot}}
\def\QHS{\Q\text{-HS}}
\def\tr{\text{tr}}
\def\op{\text{op}}
\def\MT{\text{MT}}
\def\mt{\mathfrak{mt}}
\def\so{\mathfrak{so}}
\def\su{\mathfrak{su}}
\def\Sym{\text{Sym}}
\def\Alt{\text{Alt}}
\def\dim{\text{dim}}
\def\Iso{\text{Iso}}
\def\Nrd{\text{Nrd}}
\def\Lag{\text{Lag}}
\def\ba{\overline{\phantom{x}}}
\def\disc{\text{disc}}

\begin{document}
\maketitle
\newtheorem{theorem}{Theorem}[section]
\newtheorem{corollary}[theorem]{Corollary}
\newtheorem{lemma}[theorem]{Lemma}
 
\theoremstyle{definition}
\newtheorem{definition}[theorem]{Definition}
\newtheorem{example}[theorem]{Example}

\theoremstyle{remark}
\newtheorem{remark}[theorem]{Remark}

Completing earlier work by Albert, Shimura found
all the possible endomorphism algebras (tensored
with the rationals) for complex abelian varieties of a given dimension
\cite[Theorem 5]{Shimura}. In five exceptional cases, every abelian variety
on which
a certain algebra acts has ``extra endomorphisms'', so that the full
endomorphism algebra is bigger than expected.

Complex abelian varieties $X$ up to isogeny are equivalent to polarizable
$\Q$-Hodge structures of weight 1, with Hodge numbers $(n,n)$ (where
$n$ is the dimension of $X$). In this paper, we generalize Shimura's
classification to determine all the possible endomorphism algebras
for polarizable $\Q$-Hodge structures of type $(n,0,\ldots,0,n)$.
For Hodge structures of odd weight, the answer is the same as
for abelian varieties. For Hodge structures of even weight, the answer
is similar but not identical. The proof combines ideas from Shimura
with Green-Griffiths-Kerr's approach to computing Mumford-Tate
groups \cite[Proposition VI.A.5]{GGK}.

As with abelian varieties, the most interesting feature of the classification
is that in certain cases, every Hodge structure
on which a given algebra acts
must have extra endomorphisms. (Throughout this discussion,
we only consider polarizable Hodge structures.)
One known case (pointed out to me by
Beauville) is that every $\Q$-Hodge
structure of type $(1,0,1)$ has endomorphisms by an imaginary
quadratic field and hence is of complex multiplication (CM) type,
meaning that its Mumford-Tate group is commutative. More generally,
every $\Q$-Hodge structure of type $(n,0,n)$ with endomorphisms
by a totally real field $F$ of degree $n$ has endomorphisms
by a totally imaginary quadratic extension field of $F$, and hence
is of CM type.
Another case, which seems to be new,
is that a $\Q$-Hodge structure $V$ of type $(n,0,n)$ with endomorphisms
by a CM field $F_0$ of degree $n$ must have endomorphisms by a quaternion
algebra over the totally real subfield $F$ of $F_0$. In this case, $V$
need not be of CM type; there is a period space isomorphic
to $(\C\P^1)^{[F:\Q]}$
of Hodge structures of this type, whereas there are only countably
many Hodge structures of CM type.

To motivate the results of this paper on endomorphism algebras,
consider the geometric origin of Hodge structures.
A Hodge structure {\it comes from geometry }if it is a 
summand of the cohomology
of a smooth complex projective variety defined by an algebraic
correspondence.
Griffiths found (``Griffiths transversality'') that a family
of Hodge structures
coming from geometry can vary only in certain directions, expressed
by the notion
of a variation of Hodge structures \cite[Theorem 10.2]{Voisinbook}.
In particular, any variation of Hodge structures of type $(n,0,\ldots,0,n)$
of weight at least 2 (so there is at least one $0$)
is locally constant; more generally, this holds whenever there are no two
adjacent nonzero Hodge numbers. This has the remarkable consequence
that only countably many Hodge structures of type $(n,0,\ldots,0,n)$
of weight at least 2
come from geometry. Very little is known about this countable subset
of the period domain of all Hodge structures.

One way to produce a Hodge structure of type $(n,0,n)$ is from
a smooth complex projective surface
$X$ with maximal Picard number, meaning that the Picard number is equal
to $h^{1,1}(X)$. (Then $H^2(X,\Q)$ modulo the subspace of Hodge classes,
or equivalently of divisors, is a Hodge structure of type $(p_g(X),0,p_g(X))$
that comes from geometry.) A recent survey of surfaces with maximal
Picard number is Beauville \cite{Beauville}. Many of these examples
give Hodge structures of CM type. For example, since
all $\Q$-Hodge structures of type $(1,0,1)$ are of CM type,
all complex
K3 surfaces with Picard number 20 give Hodge structures (with Hodge numbers
$(1,0,1)$) that are of CM type.

Thus, one might ask whether all Hodge structures of type $(n,0,n)$
that come from geometry are of CM type. The answer is almost certainly
no. Indeed,
a classical modular form $f$ (more precisely, a normalized eigenform)
of weight $k\geq 2$ and level $N$
determines a motive over
$\Q$ with coefficients in the field $E=\Q(f)$ of coefficients
of $f$ \cite{Scholl}. This motive has weight $k-1$
and Hodge numbers $(1,0,\ldots,0,1)$. In particular, a modular form
of weight 3 and some level $N$ determines an $E$-Hodge structure
with Hodge numbers $(1,0,1)$, and hence a $\Q$-Hodge structure
with Hodge numbers $(n,0,n)$, where $n=[E:\Q]$.
(Explicitly, this motive occurs
in $H^2$ of the elliptic modular surface of level $N$.)

Ribet explained how to check from the coefficients of a modular form
whether the associated Galois representation $\Gal(\overline{\Q}/\Q)
\arrow GL(2,E\otimes_{\Q}\Q_l)$
is of CM type, meaning that the image of the representation
has an open abelian subgroup \cite{Ribet}. From Stein's tables of modular
forms, one can read off many forms which are not of CM type,
such as the unique newform of weight 3 and level 9, with $E=\Q(\sqrt{-3})$
\cite{Stein}.
It would follow from the Hodge conjecture,
or from the weaker conjecture that every Hodge cycle is absolute
Hodge, that the associated $E$-Hodge structure of type $(1,0,1)$
is not of CM type.
Without conjectures, it is an attractive problem, not addressed
here, to show that
this Hodge structure (which comes from
geometry) is not of CM type. The problem amounts to proving the irrationality
of a suitable period of the given modular form.

All this motivates the topic of this paper: the unexpected
symmetries of Hodge structures of type $(n,0,\ldots,
0,n)$ for a number
field $E$. Several examples of ``extra'' endomorphisms
in our classification
were suggested by Ribet's analysis of the Galois representation
associated to a modular form; in those cases, the extra
endomorphisms come from algebraic cycles \cite{Ribet}.
In particular, the quaternionic
structure on the motive of a form of odd weight
comes from the outer automorphism of the group $\Gamma_1(N)$
given by the ``$W$-operator" $W=\begin{pmatrix}0& -1\\ N& 0\end{pmatrix}$.
It would be interesting
to show that the extra endomorphisms of Hodge structures which
we construct come from
algebraic cycles for other Hodge structures of type $(n,0,\ldots,0,n)$
that come from geometry, as the Hodge conjecture would predict.

The results of this paper have some of the flavor of the Kuga-Satake
construction, which shows
that all polarizable $\Q$-Hodge structures of type $(1,b,1)$
are in the tensor category generated by the cohomology of curves (or,
equivalently, the cohomology of abelian varieties).
But in fact the situation of this paper is very different.
Namely, a $\Q$-Hodge structure of type $(n,0,\ldots,0,n)$ of weight at least
2 which is not of CM type
is not in the tensor category generated by the cohomology of curves
(Corollary \ref{notcurve}).

I thank Arnaud Beauville and Tony Scholl for useful conversations.
This work was supported by NSF grant DMS-1303105.

\section{Notation}
\label{notation}

A {\it $\Q$-Hodge structure }$V$
is a rational vector space of finite dimension together with a decomposition
of $V_{\C}:=V\otimes_{\Q}\C$ as a direct sum of complex linear subspaces
$V^{a,b}$ for integers $a,b$ such that $\overline{V^{a,b}}=V^{b,a}$
and such that the grading by $a+b$, called the weight grading,
is defined over $\Q$.
A reference on Hodge structures is Voisin \cite[Chapter 7]{Voisinbook}.
A Hodge structure is pure of {\it weight }$m$ if $V^{a,b}=0$ for
$a+b\neq m$.
Hodge structures can also be defined in terms of the {\it Hodge
filtration }$F^j(V_{\C})=\oplus_{a\geq j,b\in\Z} V^{a,b}$. A smooth
complex projective variety $X$ has a Hodge structure of weight
$m$ on $H^m(X,\Q)$, for any $m$. The {\it Tate Hodge structure }$\Q(j)$
for an integer $j$ is $V=\Q$ with $V_{\C}=V^{-j,-j}$.

A {\it polarization }of a $\Q$-Hodge structure $V$ of weight $m$
is a bilinear form $\langle, \rangle\colon V\times V\arrow \Q$
which is symmetric if $m$ is even, alternating if $m$ is odd,
and which satisfies the properties \cite[section 7.1.2]{Voisinbook}:\\
(i) $\langle V^{a,b},V^{a',b'}\rangle =0$ for $a'\neq m-a$;\\
(ii) $i^{a-b}(-1)^{m(m-1)/2}\langle x,\overline{x}\rangle >0$
for all nonzero elements $x$ of $V^{a,b}$.\\
Here we write $\langle ,\rangle$ for the complex bilinear form
on $V\otimes_{\Q} \C$ associated to the given form on $V$.
For example, an ample line bundle on a smooth complex projective
variety $X$ determines a polarization of $H^m(X,\Q)$ for all $m$.
The polarizable $\Q$-Hodge structures form a semisimple
abelian category.
In this paper, all the $\Q$-Hodge structures we consider
will be polarizable, unless stated otherwise.

A $\Q$-Hodge structure $V$ (ignoring the polarization)
can also be described as a $\Q$-vector
space with a homomorphism of $\R$-groups $R_{\C/\R}G_m\arrow GL(V_{\R})$
\cite[section 1.3]{Moonen99}. The {\it Mumford-Tate group }of a Hodge
structure $V$ is the $\Q$-Zariski closure of the image of
this homomorphism. The book \cite{GGK} uses ``Mumford-Tate group''
for a slightly smaller group which we call the {\it Hodge group}:
the $\Q$-Zariski closure of the circle group $\ker(N\colon R_{\C/\R}G_m
\arrow G_m)\arrow GL(V_{\R})$
\cite[section 1.11]{Moonen99}. For example, if $V$ is pure of nonzero
weight, then the Mumford-Tate group is the product in $GL(V)$
of the Hodge group with the group $G_m$ of scalars.
The Mumford-Tate group of a polarizable Hodge structure $V$
is a connected reductive
group over $\Q$; in a sense, it describes the complexity of a Hodge
structure.

For a polarized $\Q$-Hodge structure $V$, the endomorphism algebra
$A=\End_{\QHS}(V)$ is a $\Q$-algebra
with an involution $f\mapsto \overline{f}$ given by
$$\langle fx,y\rangle = \langle x,\overline{f}y\rangle.$$
This is called the {\it Rosati involution}. The Rosati involution
is {\it positive }in the sense that $A$ has finite dimension
as a $\Q$-vector space and the reduced trace
$\tr_{A/\Q}(x\overline{x})$ is positive
for all nonzero $x$ in $V$ \cite[Remark 1.20]{Moonen99}.
It follows, for example, that if $A$
is a field, then it must be either totally real or else 
a CM field (a totally imaginary quadratic extension
of a totally real number field), and the involution must be
complex conjugation.  A convenient reference
on algebras with positive involution, in connection
with endomorphisms of abelian varieties, is Mumford
\cite[section 21]{Mumford}.

For a central simple algebra $A$ over a field $F_0$ with involution $\ba$,
define
\begin{align*}
\Sym(A,\ba)&=\{x\in A: \overline{x}=x\}\\
\Alt(A,\ba)&=\{x\in A: \overline{x}=-x\}.
\end{align*}
Write $[L:F_0]=q^2$, and let $F$ be the subfield of $F_0$
fixed by the involution.
Following the Book of Involutions, the involution on $L$
is said to be {\it orthogonal }if $F_0=F$ and $\dim_{F}\Sym(L,\ba)
=q(q+1)/2$, {\it symplectic }if $F_0=F$ and $\dim_{F}\Sym(L,\ba)
=q(q-1)/2$, and {\it unitary }if $F_0\neq F$
\cite[Proposition I.2.6]{KMRT}. These are the only
possibilities.

By Albert, every division algebra $L$ with positive involution
falls into one of four types \cite[section 21]{Mumford}. 
Type I:
$L$ is equal to $F$, a totally real field.
Type II: $L$ is a totally indefinite
quaternion algebra over a totally real field $F$, with an orthogonal
involution. (``Totally indefinite''
means that $L$ is split at every real place of $F$.)
Type III: $L$ is a totally definite quaternion algebra over a totally
real field $F$, with a symplectic involution.
Type IV: $L$ is a central
simple algebra of degree $q$ over a CM field $F_0$, and the involution
on $F_0$ is complex conjugation.

Let $(,)$ be a symmetric bilinear form on a vector space $V$ 
of dimension $n$ over
a field $F$. The {\it determinant }of the form is
$$\det(V):=\det((e_i,e_j))_{1\leq i,j\leq n}\in F^*/(F^*)^2,$$
for any basis $e_1,\ldots,e_n$ for $V$. The {\it discriminant }
of the form is the signed determinant:
$$\disc(V)=(-1)^{n(n-1)/2}\det(V)\in F^*/(F^*)^2.$$
For a central simple algebra $B$ of even degree $n=2m$
over a field $F$ with orthogonal involution,
the Book of Involutions defines
the {\it determinant }as the reduced norm of any alternating unit,
$\det(B,\ba)=\Nrd^A_F(a)\in F^*/(F^*)^2$
for $a\in \Alt(A,\ba)\cap A^*$. The {\it discriminant }is the signed
determinant:
$$\disc(B,\ba)=(-1)^m\det(B,\ba)\in F^*/(F^*)^2.$$
For a vector space $V$ of even dimension over a field $F$ with a symmetric
bilinear form, the discriminant of $V$ is equal to the discriminant
of the adjoint involution on $\End_F(V)$.

Let $E$ be a number field. We define an
{\it $E$-Hodge structure }$V$ to be a $\Q$-Hodge
structure together with a homomorphism $E\arrow \End_{\QHS}(V)$
of $\Q$-algebras.

Let $K$ be a number field which is either totally real
or a CM field. Write $a\mapsto \overline{a}$ for the involution of $K$
given by complex conjugation, which is the identity
if $K$ is totally real. A {\it polarized $K$-Hodge structure }$V$ means
a polarized $\Q$-Hodge structure $V$ together with a homomorphism
$K\arrow \End_{\QHS}(V)$ of $\Q$-algebras with involution. That is,
the form $\langle,\rangle\colon V\times V\arrow \Q$ satisfies
$$\langle ax,y\rangle=\langle x,\overline{a}y\rangle$$
for all $a$ in $K$ and $x,y$ in $V$. There does not seem to be
a reasonable notion of a polarized $E$-Hodge structure for a number
field $E$ which is not totally real or a CM field, although one could
require the underlying $\Q$-Hodge structure to be polarizable.

Let $E$ be a totally real or CM field,
and let $V$ be a polarized $E$-Hodge structure of weight $m$.
Then there is a unique
$(-1)^m$-hermitian form $(,)\colon V\times V\arrow E$ such that
$\langle x,y\rangle=\tr^E_{\Q}(x,y)$. By a $(-1)^m$-hermitian
form, we mean that $(ax,y)=a(x,y)$, $(x,ay)=\overline{a}(x,y)$,
and $(x,y)=(-1)^m\overline{(y,x)}$ for $x,y\in V$ and $a\in E$;
thus $(,)$ is a bilinear form if $E$ is totally real.
The existence and uniqueness of $(,)$ follow by observing
that for $x,y$ in $V$, $(x,y)$ must be the unique element $u\in L$
such that $\langle ax,y\rangle=\tr^E_{\Q}(au)$ for all $a\in L$.
This uniquely determines $u$, because $a,b\mapsto \tr^L_{\Q}(ab)$
is a nondegenerate bilinear form on $L$ as a $\Q$-vector space.

A $\Q$-Hodge structure $V$ is of {\it CM type} if it is polarizable
and its Mumford-Tate
group is commutative. In particular, if there is a CM field $K$
such that $V$ is a $K$-Hodge structure and $\dim_KV=1$, then
$V$ is of CM type \cite[Proposition V.3]{GGK}. There
are only countably many isomorphism classes of Hodge structures of CM type.
They all come from geometry, in fact (up to Tate twists) from
the rational cohomology of abelian varieties with complex multiplication,
by Serre. More strongly, by Abdulali,
every effective Hodge structure
of CM type occurs in the cohomology of some abelian variety
with complex multiplication, with no Tate twist needed
\cite{Abdulali}.

We say that a $\Q$-Hodge structure $V$ is of {\it type }$(a_0,\ldots,a_m)$
if $V$ has weight $m$ and 
the only nonzero Hodge numbers are $\dim_{\C}V^{j,m-j}=a_j$
for $0\leq j\leq m$.
Let $V$ be a $\Q$-Hodge structure of weight 2 and type $(n,0,n)$,
the main
situation considered in this paper. Then the bilinear form
$\langle,\rangle$ on $V$ is positive definite. Conversely,
for a $\Q$-vector space $V$ of dimension $2n$
with a positive definite symmetric
bilinear form $\langle,\rangle$, a Hodge structure of type
$(n,0,n)$ on $(V,\langle,\rangle)$
is equivalent to an isotropic $\C$-linear subspace
$V^{2,0}\subset V\otimes_{\Q}\C$ of dimension $n$. (The positivity property
$\langle x,\overline{x}\rangle >0$ for nonzero $x$ in $V^{2,0}$
is automatic; in fact, $\langle x,\overline{x}\rangle >0$
for all nonzero $x$ in $V\otimes_{\Q}\C$.) Therefore, the period domain
of Hodge structures of type $(n,0,n)$ on $(V,\langle,\rangle)$
is the isotropic Grassmannian $\Gr_{\isot}(n,2n)$ over $\C$.

Let $E$ be a number field which is totally real or a CM field,
and let $r=[E:\Q]$.
Let $V$ be an $E$-Hodge structure. Then $V^{a,b}$ is an
$E\otimes_{\Q}\C$-module for each $a,b\in\Z$. The ring
$E\otimes_{\Q}\C$ is isomorphic to a product of copies of $\C$,
indexed by the embeddings $\sigma_1,\ldots,\sigma_r$ of $E$
into $\C$. Therefore, the complex vector space
$V^{a,b}$ splits as a direct sum indexed
by the embeddings $\sigma_l$ (the subspace where $E$ acts through its
embedding $\sigma_l$ in $\C$). 
We say that an $E$-Hodge structure $V$ is of {\it type }$(a_0,
\ldots,a_n)$ if, for each embedding $\sigma_l\colon E\inj \C$,
the summand of $V^{j,n-j}$ corresponding to $\sigma_l$
has complex dimension $a_j$. It follows that, as a $\Q$-Hodge
structure, $V$ has type $(ra_0,\ldots,ra_n)$.

In general, an $E$-Hodge structure need not be of a single
type $(a_0,\ldots,a_n)$
in this sense. For example, if $X$ is an elliptic curve with complex
multiplication by an imaginary quadratic field $K$, then $V=H^1(X,\Q)$ can be
viewed as a $K$-Hodge structure, of dimension 1 as a $K$-vector
space. This $K$-Hodge structure has Hodge numbers $(1,0)$ under
one complex embedding of $K$ and Hodge numbers $(0,1)$ under the conjugate
embedding.

\section{Polarizations}

\begin{lemma}
\label{polar}
Let $K$ be a number field which is either totally real
or a CM field. Let $V$ be a $K$-Hodge structure such that
the underlying $\Q$-Hodge structure is polarizable.
Then $V$ is polarizable as a $K$-Hodge structure.
\end{lemma}

\begin{proof}
We can assume that $V$ is pure of some weight $m$.
Let $\langle,\rangle$ be a polarization of $V$ as a $\Q$-Hodge
structure. We have to produce another polarization $\langle,\rangle_2\colon
V\times V\arrow \Q$ such that $\langle ax,y\rangle_2=\langle x,\overline{a}y
\rangle_2$ for all $a$ in $K$ and $x,y$ in $V$. Here $\overline{a}$ denotes
complex conjugation on $K$, which is the identity if $K$ is totally
real.

For $a,b$ in $K$, define $\langle a,b\rangle=\tr^K_{\Q}(a\overline{b})$.
This is a positive definite symmetric bilinear form on $K$ as a
$\Q$-vector space. In what follows, write $K^*$ for the dual of $K$
as a $\Q$-vector space. Then the form we defined on $K$ gives an
identification of $K$ with $K^*$. As a result, the identity map
$1_K\in K^*\otimes_{\Q}K$ corresponds to a canonical element
$B\in K\otimes_{\Q}K$. We can write $B$ explicitly in terms of a
basis $e_1,\ldots,e_r$ for $K$ as a $\Q$-vector space. Let 
$f_1,\ldots,f_r$ be the dual basis for $K$, meaning that
$\tr^K_{\Q}(e_i\overline{f_j})=\delta_{i,j}$ for all $i,j$.
Then $B=\sum_j f_j\otimes e_j$.

We use $B$ to ``average'' the given polarization on $V$.
Explicitly, define
$$\langle x,y\rangle_2=\sum_{j=1}^r\langle f_jx,e_jy\rangle.$$
We want to show that $\langle ux,y\rangle_2=\langle x,\overline{u}y
\rangle_2$ for all $u$ in $K$ and $x,y$ in $K$. That is, we have to show
that $\sum_j \langle f_jux,e_jy\rangle=\sum_j\langle f_jx,e_j\overline{u}y
\rangle$. It suffices to show that $\sum_j f_ju\otimes e_j
=\sum_j f_j\otimes\overline{u}e_j$ in $K\otimes_{\Q}K$.
We can identify $K\otimes_{\Q}K$ with $\Q^{r^2}$ as a $\Q$-vector
space by pairing the first variable with $e_j$ and the second variable
with $f_k$, for any given $j,k\in\{1,\ldots,r\}$, using the bilinear
form on $K$. Thus it suffices to show that
$$\langle e_j,f_k\overline{u}\rangle=\langle \overline{u}e_j,f_k\rangle
\in \Q$$
for all $j$ and $k$. This is true, since the left side is
$\tr^K_{\Q}(e_j\overline{f_ku})$ and the right side is
$\tr^K_{\Q}(\overline{u}e_j\overline{f_k})$.

It remains to check that $\langle ,\rangle_2$ is a polarization
of $V$ as a $\Q$-Hodge structure, using that $\langle,\rangle$
is a polarization. First, the formula $B=\sum_j f_j\otimes e_j$
for the tensor $B$ above works using any basis for $K$ as a $\Q$-vector
space in place of $e_1,\ldots,e_r$ and the dual basis in place of 
$f_1,\ldots,f_r$. In particular, $B$ can also be written
as $B=\sum_j e_j\otimes f_j$. From that it is clear that $\langle,
\rangle_2$ is $(-1)^m$-symmetric, since $\langle,\rangle$ is 
$(-1)^m$-symmetric.

Since the action of $K$ on $V$
sends each subspace $V^{a,b}$ of $V\otimes_{\Q}\C$ into itself,
the definition of $\langle,\rangle_2$ shows that we have
$\langle V^{a,b},V^{a',b'}\rangle_2 =0$ for $a'\neq m-a$.
To prove the positivity property of $\langle,\rangle_2$, it is
convenient to choose an orthogonal basis $e_1,\ldots,e_r$ for $K$
as a $\Q$-vector space. Then $a_j:=\langle e_j,e_j\rangle\in \Q$
is positive, and the dual basis for $K$ is given by $f_j=e_j/a_j$.
So
$$\langle x,y\rangle_2=\sum_{j=1}^r \frac{1}{a_j}\langle e_jx,e_jy\rangle$$
for $x,y$ in $V$. That implies the same identity for the associated
complex bilinear form $\langle,\rangle_2$ on $V\otimes_{\Q}\C$.
It follows that
$$i^{a-b}(-1)^{m(m-1)/2}\langle x,\overline{x}\rangle_2 >0$$
for all nonzero elements $x$ of $V^{a,b}$, from the corresponding
inequality for $\langle,\rangle$.
(Note that $\overline{e_jx}=e_j\overline{x}$
for $x$ in $V\otimes_{\Q}\C$,
because $e_j\in K$ is a $\Q$-linear endomorphism of $V$.)
\end{proof}

\section{Endomorphism algebras}

In the following theorem and proof, we follow Shimura's
notation where possible \cite[Theorem 5]{Shimura}. In particular,
for a division algebra $L$ and a subfield $K$ of the center of $L$,
write $[L:K]$ for the dimension of $L$ as a $K$-vector space.

Since the abelian category of polarizable $\Q$-Hodge structures is semisimple,
the endomorphism algebras of all polarizable $\Q$-Hodge structures
of type $(n,0,\ldots,0,n)$ are determined if we can find the endomorphism
algebras of all simple $\Q$-Hodge structures of that type (including
smaller values of $n$). That is solved by Theorem \ref{main}.

Let $V$ be a simple polarizable $\Q$-Hodge structure of weight $w\geq 1$
with Hodge numbers $(n,0,\ldots,0,n)$.
The endomorphism algebra $L$
of $V$ is a division algebra with positive involution. We use
Albert's classification of $L$ into Types I through IV (section
\ref{notation}). Let
$F_0$ be the center of $L$, which is a CM field or a totally real
field, and let $F$ be the subfield of $F_0$ fixed by complex
conjugation. 

Write $g=[F:\Q]$, $2n=m[L:\Q]$,
and $q^2=[L:F_0]$. For $V$ of Type IV,
$L\otimes_{\Q}\C$ is isomorphic to the product of $2g$ copies
of $M_q(\C)$.
Write the simple $L\otimes_{\Q}\C$-modules, each of complex dimension $q$,
as $\chi_1,\ldots,\chi_g,
\overline{\chi_1},\ldots,\overline{\chi_g}$. Let $r_{\nu}$
and $s_{\nu}$ be the multiplicities of $\chi_{\nu}$
and $\overline{\chi_{\nu}}$, respectively,
in the representation of $F_0$
on $V^{2,0}\subset V\otimes_{\Q}\C$. Then $r_{\nu}+s_{\nu}=mq$
for $\nu=1,\ldots,g$. 

As in Shimura, the proof does something more precise
than determining the possible endomorphism algebras.
Rather, for each division algebra $L$ with positive involution,
we describe the Mumford-Tate domain $D$ of $\Q$-Hodge structures
with a given bilinear form and a given action of $L$. For each connected
component of $D$, we determine whether a very general $\Q$-Hodge structure
in that component has endomorphism algebra equal to $L$ or bigger than
$L$. In Type IV, the components of $D$ are indexed
by the numbers $r_{\nu}$ and $s_{\nu}$ defined above.

\begin{theorem}
\label{main}
Let $V$ be a simple polarizable $\Q$-Hodge structure of weight $w\geq 1$
with Hodge numbers $(n,0,\ldots,0,n)$. Let $L$ be the endomorphism
algebra of $V$, and define $F_0$ and $F$ as above.
Then $[L:\Q]$ divides $2n$ and $[F:\Q]$ divides $n$.

Conversely, every division algebra with positive involution
satisfying these two bounds is the endomorphism algebra
of some simple polarizable $\Q$-Hodge structure of weight $w$
and Hodge numbers $(n,0,\ldots,0,n)$, except for five cases
when $w$ is odd and six cases when $w$ is even, as follows.

Odd weight exceptional cases:

(a) Type III and $m=1$. Then the $\Q$-Hodge structure
$V$ is isomorphic to a direct sum $W^{\oplus 2}$, where $W$
has endomorphisms by a CM quadratic extension $F_0$ of $F$ and
$\dim_{F_0}(W)=1$. In particular, $V$ is of CM type.

(b) Type III, $m=2$, $\disc(B,\ba)=1$ in $F^*/(F^*)^2$,
where $B$ is the centralizer of $L$ in $\End_{F}(V)$ and $\ba$
is its involution, coming from the $L$-invariant symmetric bilinear
form $\langle,\rangle$ on $V$.
In all but 2 of the $2^g$ connected components of the Mumford-Tate domain
of $L$-invariant Hodge structures on $(V,\langle,\rangle)$,
a generic $\Q$-Hodge structure $V$ has the ``expected'' endomorphism
algebra $L$. In the other 2 components, a generic $\Q$-Hodge structure
$V$ is a direct sum $W^{\oplus 2}$,
where $W$ is simple and has endomorphism algebra
a Type II quaternion algebra over $F$.

(c) Type IV and $\sum_{\nu=1}^g r_{\nu}s_{\nu}=0$.
Then the $\Q$-Hodge structure $V$ is isomorphic to a direct sum
$W^{\oplus mq^2}$ such that $W$ has endomorphisms by the CM field $F_0$
and $\dim_{F_0}(W)=1$. In particular, $V$ is of CM type.

(d) Type IV, $m=2$, $q=1$, $r_{\nu}=s_{\nu}=1$ for $\nu=1,\ldots,
g$. Then $V$ is generically simple, with
endomorphism algebra a Type II
quaternion algebra over $F$.

(e) Type IV, $m=1$, $q=2$, $r_{\nu}=s_{\nu}=1$ for $\nu=1,\ldots,
g$. Then $V$ is isomorphic to the direct sum $W^{\oplus 2}$,
where $W$ is generically simple, with endomorphism algebra a Type II
quaternion algebra over $F$.

Even weight exceptional cases:

(a') Type II and $m=1$. Then the $\Q$-Hodge structure $V$
is isomorphic
to a direct sum $V=W^{\oplus 2}$, where
$W$ has endomorphisms by a CM quadratic extension $F_0$ of $F$,
and $\dim_{F_0}W=1$. In particular,
$V$ is of CM type.

(b') Type II, $m=2$, $\disc(B,\ba)=1$ in $F^*/(F^*)^2$,
where $B$ is the centralizer of $L$ in $\End_{F}(V)$ and $\ba$
is its involution, coming from the $L$-invariant symmetric bilinear
form $\langle,\rangle$ on $V$.
In all but 2 of the $2^g$ connected components of the Mumford-Tate domain
of $L$-invariant Hodge structures on $(V,\langle,\rangle)$,
a generic $\Q$-Hodge structure $V$ has the ``expected'' endomorphism
algebra $L$. In the other 2 components, a generic $\Q$-Hodge structure
$V$ is a direct sum $W^{\oplus 2}$,
where $W$ is simple and has endomorphism algebra
a Type III quaternion algebra over $F$.

(c') Type IV and $\sum_{\nu=1}^g r_{\nu}s_{\nu}=0$.
Then $V$ is a direct sum
$V=W^{\oplus mq^2}$ for a $\Q$-Hodge structure $W$ with endomorphisms
by $F_0$ such that $\dim_{F_0}(W)=1$. So $V$ is of CM type.

(d') Type IV, $m=2$, $q=1$, $r_{\nu}=s_{\nu}=1$ for $\nu=1,\ldots,g$.
Then $V$ generically has endomorphism algebra 
a Type III quaternion algebra over $F$.

(e') Type IV, $m=1$, $q=2$, $r_{\nu}=s_{\nu}=1$ for $\nu=1,\ldots,g$.
Then the $\Q$-Hodge structure $V$ is a direct sum $V=W^{\oplus 2}$,
and $W$ generically has endomorphism algebra
a Type III quaternion algebra over $F$.

(f') Type I and $m=2$. Then the $\Q$-Hodge structure
$V$ has endomorphisms by a CM quadratic
extension $F_0$ of $F$. Since $\dim_{F_0}(V)=1$, $V$ is of CM type.
\end{theorem}

\begin{proof}
At first, we consider a more general situation.
Let $V$ be any simple polarizable $\Q$-Hodge structure.
Fix a polarization $\langle,\rangle\colon V\times V\arrow \Q$.
Then $L:=\End_{\QHS}(V)$ is an algebra with positive involution,
by section \ref{notation}.
Since $V$ is
a vector space over the division algebra $L$,
the dimension $[L:\Q]$ divides $\dim_{\Q}V$. When $V$ has Hodge
numbers $(n,0,\ldots,0,n)$, this proves that $[L:\Q]$ divides $2n$.

Let $F_0$ be the center of $L$, which is a totally real field
or a CM field, and let $F$ be the subfield of $F_0$ fixed
by complex conjugation (which is also the restriction
of the involution on $L$). Let $g=[F:\Q]$. Then $F$ is totally real,
and so $F\otimes_{\Q}\C$ is the product of copies of $\C$
indexed by the embeddings $\sigma_1,\ldots,\sigma_g\colon
F\inj \R$. Each summand $V^{b,c}$ of $V\otimes_{\Q}\C$ 
is a module over $F\otimes_{\Q}\C$. So $V^{b,c}$
splits as a direct sum of complex linear subspaces
on which $F$ acts by $\sigma_1,\ldots,\sigma_g$, respectively.

For any integers $b$ and $c$,
$V^{b,c}$ is the complex conjugate of $V^{c,b}$ in
$V\otimes_{\Q}\C$. Let $x$ be an an element of $V^{b,c}$ on which
$L$ acts by an embedding $\sigma_j$. Since each element $a$ in $F$
acts $\Q$-linearly on $V$, we have $a(\overline{x})=\overline{ax}
=\overline{\sigma_j(a)x}=\sigma_j(a)\overline{x}$ for all $a$ in $F$,
where the last
equality uses that $\sigma_j(a)$ is real. So each embedding
$L\inj\R$ occurs with the same multiplicity in $V^{b,c}$
as in $V^{c,b}$.
Also, $V$ is a free $F$-module, and so $V\otimes_{\Q}\C$
is a free $F\otimes_{\Q}\C$-module. That is, each embedding
$F\inj \R$ occurs the same number of times in $V\otimes_{\Q}\C$.

We now make our assumption that $V$ has weight $w$ and Hodge numbers
$(n,0,\ldots,0,n)$. (To prove the following bound,
it would suffice to assume that $V^{w/2,w/2}=0$.)
Then the previous paragraph implies that each embedding 
$F\inj\R$ occurs the same number of times in $V\otimes_{\Q}\C\cong \C^{2n}$,
and this number is even.
Therefore, $[F:\Q]$ divides $n$, which proves the first part
of the theorem. 

For any positive integer $w$, the category of $\Q$-Hodge structures
of weight 1 and Hodge numbers $(n,n)$ can be identified with
the category of $\Q$-Hodge structures of weight $w$ and Hodge numbers
$(n,0,\ldots,0,n)$, just by renaming $V^{1,0}\subset V\otimes_{\Q}\C$
as $V^{w,0}$. For even weights $w$, this equivalence does not
preserve polarizability and hence is of little interest.
But for odd weights $w$, this equivalence does preserve 
polarizability. Therefore, the endomorphism
algebras of the simple polarizable $\Q$-Hodge structures
of odd weight $w$ and Hodge numbers $(n,0,\ldots,0,n)$
are the same as the endomorphism algebras of the simple abelian
varieties of dimension $n$. These were determined by Shimura
\cite[Theorem 5]{Shimura},
giving the answer in the theorem.

Our proof in even weight is analogous to Shimura's argument,
but we use the language of Mumford-Tate groups so that fewer
explicit calculations are required. (The reader could apply the same
method to reprove Shimura's classification.)

There is an equivalence of categories between $\Q$-Hodge structures
with Hodge numbers $(n,0,n)$ and $\Q$-Hodge structures
of any even weight $w=2m$ and Hodge numbers $(n,0,\ldots,0,n)$,
just by renaming $V^{2,0}\subset V\otimes_{\Q}\C$
as $V^{2m,0}$. This equivalence preserves polarizability; we just
need to replace a polarization $\langle,\rangle$ on $V$ of weight 2
by $(-1)^m\langle,\rangle$. Therefore, the same
endomorphism algebras occur in any even weight. 

It would be easy
to argue directly with $V$ of any even weight, but we choose
to work with $V$ of weight 2 and Hodge numbers
$(n,0,n)$. In this case, the polarization $\langle, \rangle$ of $V$ is 
a positive definite symmetric bilinear form on the
$\Q$-vector space $V$, by section \ref{notation}.
For each positive definite symmetric bilinear form
$\langle ,\rangle$ on $V$,
the space of Hodge structures on $(V,\langle,\rangle)$ of
type $(n,0,n)$ is the space $\Gr_{\isot}(n,2n)$ of 
all isotropic $n$-dimensional $\C$-linear subspaces in $V\otimes_{\Q}\C
\cong\C^{2n}$.

Let $L$ be a division algebra with positive involution.
Let $F_0$ be the center of $L$, and let $F$ be the subfield
of $F_0$ fixed by complex conjugation, or equivalently by
the involution on $L$. Assume
that $[L:\Q]$ divides $2n$ and $[F:\Q]$ divides $n$. Let $V$
be a $\Q$-vector space of dimension $2n$. By the first
assumption, we can give $V$ the structure of a left $L$-vector space; choose
such an action of $L$ on $V$. Then there is a positive definite
symmetric bilinear form $\langle,\rangle\colon V\times V\arrow \Q$ which
is $L$-invariant, meaning that
$\langle ux,y\rangle=\langle x,\overline{u}y\rangle$
for all $u$ in $L$ and $x,y$ in $V$. Indeed, $L$ itself
has an $L$-invariant positive definite symmetric bilinear form
given by $\langle x,y\rangle=\tr^L_{\Q}(x\overline{y})$, and we can
view $V$ as the direct sum of copies of $L$.

Fix any positive definite symmetric bilinear form $\langle,\rangle$ on the
$\Q$-vector space $V$ which is $L$-invariant.
We will show that there is an $L$-invariant $\Q$-Hodge structure
on $(V,\langle,\rangle)$ with endomorphism algebra equal to $L$,
apart from the exceptions listed in the theorem. This is slightly
stronger than the theorem as stated, since we are fixing the action of $L$
and the symmetric bilinear form on $V$.

Write $g=[F:\Q]$, $2n=m[L:\Q]$ where $m$ is a positive integer
(which is even for $V$ of Type I by our assumption that
$[F:\Q]$ divides $n$),
and $q^2=[L:F_0]$. For $V$ of Type IV, recall the definition
of $r_{\nu}$ and $s_{\nu}$ for $\nu=1,\ldots,g$ from before the
theorem.

We can describe the ``Mumford-Tate domain'' of all $L$-invariant
Hodge structures on $(V,\langle,\rangle)$. It helps to observe that
the ring $L\otimes_{\Q}\R$ is isomorphic to:
\begin{align*}
\text{Type I: }&\R\times\cdots\times \R\\
\text{Type II: }&M_2(\R)\times\cdots\times M_2(\R)\\
\text{Type III: }&\HH\times\cdots\times \HH\\
\text{Type IV: }&M_q(\C)\times\cdots\times M_q(\C),
\end{align*}
where there are $g$ factors in each case, and the involution on the right
is tbe identity in Type I, $X\mapsto X^t$ on each copy of $M_2(\R)$ in Type II,
$X\mapsto \tr^{\HH}_{\R}(x)-x$ on each copy of the quaternions $\HH$
in Type III, and $X\mapsto \overline{X^t}$ on each copy of $M_q(\C)$
in Type IV \cite[pp.~201-202]{Mumford}. Recall that the period domain
of all Hodge structures on $(V,\langle,\rangle)$ is the isotropic
Grassmannian $\Gr_{\isot}(n,2n)$. We deduce that the Mumford-Tate
domain $D$ of $L$-invariant Hodge structures on $(V,\langle,\rangle)$
is:
\begin{align*}
\text{Type I: }&\Gr_{\isot}(m/2,m)^g\\
\text{Type II: }&\Gr_{\isot}(m,2m)^g\\
\text{Type III: }&\Gr_{\Lag}(m,2m)^g\\
\text{Type IV: }&\bigg[ \coprod_{j=0}^{mq}\Gr(j,mq)\bigg]^g.
\end{align*}
Here $\Gr_{\isot}(a,b)$ is the space of isotropic linear
subspaces of dimension $a$ in a complex vector space of dimension
$b$ with a nondegenerate symmetric bilinear form, and $\Gr_{\Lag}(m,2m)$
is the space of isotropic subspaces of dimension $m$
in a complex vector space of dimension $2m$ with
a nondegenerate alternating bilinear form. We see that the number
of connected components of the Mumford-Tate domain $D$
is $2^g$ in Type I, $2^g$ in Type II, 1 in Type III,
and $(mq+1)^g$ in Type IV.

Let $d$ be the complex dimension of the Mumford-Tate domain. Then:
\begin{align*}
\text{Type I: }n&=\frac{m}{2}g, \; d=\frac{1}{2}\frac{m}{2}\bigg( 
\frac{m}{2}-1\bigg) g\\
\text{Type II: }n&=2mg, \; d=\frac{1}{2}m(m-1)g\\
\text{Type III: }n&=2mg, \; d=\frac{1}{2}m(m+1)g\\
\text{Type IV: }n&=q^2mg, \; d=\sum_{\nu=1}^g r_{\nu}s_{\nu}.
\end{align*}
Shimura's formula for the dimensions of the analogous Mumford-Tate
domains in the period domain of abelian varieties is similar,
but with the expressions $x(x+1)/2$ switched with $x(x-1)/2$
\cite[section 4.1]{Shimura}. This is related
to the switch between symplectic and orthogonal groups, in comparing
polarizable Hodge structures of odd weight with those of even weight.

Let $D^0$ be a connected component of the Mumford-Tate domain $D$
of $\Q$-Hodge structures of type $(n,0,n)$ on $(V,\langle,\rangle)$
with endomorphisms by the given homomorphism $L\arrow \End_{\Q}(V)$
of algebras with involution. For each larger subalgebra $L'$ 
of $\End_{\Q}(V)$, the subspace of $D^0$
of Hodge structures  with endomorphisms by $L'$
is a closed analytic subspace of $D^0$. Therefore, there is a well-defined
algebra with involution $A\subset \End_{\Q}(V)$, the  {\it generic
endomorphism
algebra }for $D^0$, which is the endomorphism
algebra of a very general $\Q$-Hodge structure $V$ in $D^0$. (That is, $A$
is the endomorphism algebra of every Hodge structure in $D^0$ outside
countably many closed analytic subspaces not equal to $D^0$.) Clearly
$A$ contains $L$. The main part of the theorem is to show that $A$ is equal
to $L$ in most cases.

It is also convenient to consider the {\it generic Hodge group }$M$
of $D^0$, defined as the Hodge group (section \ref{notation})
of a very general $\Q$-Hodge structure in $D^0$. We know that
$A$ is the commutant of $M$ in $\End_{\Q}(V)$. Since $L$ is contained
in $A$ and $M$ is a connected $\Q$-group, $M$ is contained
in the connected component $H$ of the centralizer of $L$
in the $\Q$-group $O(V)$. We call $H$ the {\it expected Hodge group}.
(The $\Q$-group $H$ depends on $(V,\langle,\rangle,L)$,
but not on the particular
component $D^0$.)

A crucial observation is that the generic endomorphism algebra $A\subset
\End_{\Q}(V)$ and the generic Hodge group $M\subset O(V)$
are determined by $(V,\langle,\rangle,L\arrow \End_{\Q}(V),D^0)$.
Since $H(\Q)$ preserves these data, $H(\Q)$ normalizes
both $A$ and $M$; for $D^0$,
this uses that $H$ is connected. Since $H$ is a connected group
over the perfect field $\Q$, $H(\Q)$ is Zariski dense in $H$
\cite[Corollary 18.3]{Borel},
and so $A$ and $M$ are in fact normalized by the algebraic group $H$.
Since $M\subset H$, we can say that $M$ is a connected normal $\Q$-subgroup
of $H$. Thus, if $H$ is $\Q$-simple, then $M$ must be either 1 or $H$.
But $M$ can never be 1; that would mean that the generic Hodge structure
$V$ in $D^0$ has $V^{a,b}=0$ for $a\neq b$,
whereas in fact $V^{2,0}$ is not zero.
So if $H$ is $\Q$-simple, then the generic Hodge group $M$ is equal
to $H$. (This argument is inspired by Green-Griffiths-Kerr's approach
to computing generic Mumford-Tate groups, although
they exclude the non-connected period domains which we encounter here
\cite[Proposition VI.A.5]{GGK}.)
As a result, when $H$ is $\Q$-simple,
we know the generic endomorphism algebra $A$:
it is the centralizer of the ``known'' $\Q$-group $H$ in $\End_{\Q}(V)$.
In most cases, that will imply that $A$ is equal to $L$, as we want.

Suppose that $L$ is of Type I, so $L=F$. Then the expected Hodge group $H$
is $R_{F/\Q}SO(_F V)$, where $_F V$ denotes $V$ as an $F$-vector space.
Suppose that $m=\dim_F(V)$ (which is even in this case) is at least 6,
or that $m=4$ and $V$ has discriminant not equal to $1$ in $F^*/(F^*)^2$.
Then $SO(_F V)$ is $F$-simple, and so $H$ is $\Q$-simple.
By the argument above, the generic Hodge group $M$ of $\Q$-Hodge
structures in each component $D^0$ must be equal to $H$. So the generic
endomorphism algebra $A$ is equal to the centralizer
of $R_{F/\Q}SO(_F V)$ in $\End_{\Q}(V)$. Clearly $A$ contains $F$.
To show that $A$ is equal to $F$, note that $A\subset \End_{\Q}(V)$
must commute with the Lie algebra of $R_{F/\Q}SO(_F V)$, which
is $\so(_F V)$, and so it commutes with the $\Q$-algebra generated
by this Lie algebra. The $F$-algebra generated by $\so(_F V)$ is equal
to $\End_F(V)$, just using that $\dim_F(V)$ is at least 3 (so that $_F V$
is an absolutely irreducible representation of $\so(_F V)$). So $A\subset
\End_{\Q}(V)$ must be contained in the commutant of $\End_F(V)$
in $\End_{\Q}(V)$, which is equal to $F$. Thus we have shown that for $L$
of Type I with $m\geq 6$, or with $m=4$ and $V$ of discriminant
not equal to $1\in F^*/(F^*)^2$,
the generic endomorphism algebra is equal to $L$
($=F$). At the same time, we found that the generic Mumford-Tate
group is equal to $R_{F/\Q}SO(_F V)$.

For a semisimple algebra $A$ with involution $\ba$, define
the group of {\it isometries }to be
$$\Iso(A,\ba)=\{g\in A^*:\overline{g}=g^{-1}\}.$$
Following the Book of Involutions \cite{KMRT}, we write
$$\Iso(A,\ba)=\begin{cases} O(A,\ba) & \text{if }\ba
\text{ is of orthogonal type,}\\
Sp(A,\ba) & \text{if }\ba
\text{ is of symplectic type,}\\
U(A,\ba) & \text{if }\ba
\text{ is of unitary type.}
\end{cases}$$
For an algebra $A$ with orthogonal involution, with center $F$
and $[A:F]=q^2$, 
the subgroup $O^{+}(A,\ba)=\ker(\Nrd\colon O(A,\ba)\arrow
\{\pm 1\})$ (as an algebraic group) is a form of $SO(q)$ over $F$
(meaning that the two groups become isomorphic over an algebraic
closure of $F$).
For $A$ with symplectic involution, $q$ must be even, and $Sp(A,\ba)$
is a form of the symplectic group $Sp(q)$ over $F$. Finally,
if $A$ has unitary involution over $F_0$ and $[A:F_0]=q^2$,
with $F\subset F_0$ the subfield fixed by the involution,
then the unitary group $U(A,\ba)$ is a form of $GL(q)$ over $F$.

Next, let $(V,\langle,\rangle,L,D^0)$
be of Type II. Thus $L$ is a totally indefinite
quaternion algebra over a totally real field $F$, and
$2n=m[L:\Q]$. By definition, the ``expected Hodge group''
$H$ is the connected component of the centralizer of $L$
in $SO(V_{\Q})$. Thus the Lie algebra of $H$ is the antisymmetric
part of the centralizer $B$ of $L$ in $\End_{\Q}(V)$,
or equivalently in $\End_F(V)$. Here $B$
is isomorphic to $M_m(L^{\op})$, with an orthogonal involution.
So the expected Hodge group $H$ is $R_{F/\Q}O^+(B,\ba)$.
Here $O^+(B,\ba)$ is an $F$-form of $SO(2m)$.

Recall from section \ref{notation} the discriminant
of a central simple algebra $B$ of even degree $n=2m$
with orthogonal involution.
For $n=4$, the group $O^{+}(B,\ba)$
is $F$-simple if and only if $\disc(B,\ba)\neq 1
\in F^*/(F^*)^2$ \cite[Theorem 15.7 and section 26.B]{KMRT}.

Suppose that $m\geq 3$, or that $m=2$ and $\disc(B,\ba)\neq
1 \in F^*/(F^*)^2$. Then
$O^+(B,\ba)$ is $F$-simple,
and so $H$ is $\Q$-simple. By the argument above, the generic
Hodge group $M$ is equal to $H$. 

So the generic endomorphism
algebra $A$ is the centralizer of $R_{F/\Q}O^+(B,\ba)$
in $\End_{\Q}(V)$. Clearly $L$ is contained in $A$. To see that equality
holds, note that $A$ commutes with the Lie algebra $\so(B,\ba)
=\Alt(B,\ba)\subset \End_{\Q}(V)$, hence with the algebra
generated by $\Alt(B,\ba)$. This algebra is all of $B$,
as one can check over an algebraic closure $\overline{F}$ of $F$,
using that $\overline{F}^{2m}$ is an irreducible
representation of $SO(2m)$, fore $m\geq 2$. So $A$ is contained
in the commutant of $B\cong M_m(L^{\op})$ in $\End_{\Q}(V)$
or equivalently in
$\End_F(V)$, which is equal to $L$. We have shown that in Type II
with $m\geq 3$, the generic endomorphism algebra is equal to $L$,
as we want.

Next, let $(V,\langle,\rangle,L,D^0)$
be of type III. Thus $L$ is a totally definite quaternion
algebra over a totally real field $F$, and $2n=m[L:\Q]$.
Let $B$ be the centralizer of $L$ in $\End_{\Q}(V)$ or equivalently
in $\End_F(V)$; then $B\cong M_m(L^{\op})$ with a symplectic involution.
The ``expected Hodge group'' $H$ is defined to be the connected
component of the centralizer of $L$ in $SO(V_{\Q})$, that is,
the connected component of $B\cap SO(V_{\Q})$. So 
$H=R_{F/\Q}Sp(B,\ba)$. Since $Sp(B,\ba)$ is an
$F$-form of the symplectic group $Sp(2m)$, it is absolutely simple
for all $m\geq 1$. So $H$ is $\Q$-simple. After that, the argument
is the same as in Type II. We conclude that the generic Hodge group
is $R_{F/\Q}Sp(B,\ba)$ and the generic endomorphism
algebra is $L$, for $L$ of Type III with any $m\geq 1$.

Finally, let $(V,\langle,\rangle,L,D^0)$
be of Type IV.  Thus $L$ is a central division algebra
over a CM field $F_0$, $L$ has a unitary involution, and $2n=m[L:\Q]$. 
We write $[L:F_0]=q^2$. Let $B\cong M_m(L^{\op})$ be the centralizer
of $L$ in $\End_{\Q}(V)$, or equivalently in $\End_{F_0}(V)$.
The ``expected Hodge group'' $H$
is defined as the connected component of the centralizer of $L$
in $SO(V_{\Q})$, that is, connected component of the identity
in $B\cap SO(V_{\Q})$.
So $H=R_{F/\Q}U(B,\ba)$. Since $U(B,\ba)$
is an $F$-form of $GL(mq)$, Type IV is more subtle, in that
$H$ is never $\Q$-simple. It is the product of the $\Q$--simple
group $R_{F/\Q}SU(B,\ba)$ with a torus.

The dimension of the Mumford-Tate domain $D^0$ is
$\sum_{\nu=1}^g r_{\nu}s_{\nu}$. Suppose that this is positive.
Since there are only countably many Hodge structures
of CM type, a very general Hodge structure $V$ in $D^0$ is not
of CM type. So the generic Hodge group is not commutative.
Since the generic Hodge group is normal in $R_{F/\Q}U(B,\ba)$,
it must contain $R_{F/\Q}SU(B,\ba)$.
since this is $\Q$-simple for all $m\geq 1$. The generic
endomorphism algebra $A$ contains $L$, and is
contained in the centralizer of $R_{F/\Q}SU(B,\ba)$
in $\End_{\Q}(V)$. So $A$ must commute with the Lie algebra
$\su(B,\ba)=\ker(\tr^B_{F_0}\colon \Alt(B,\ba)\arrow F_0)
\subset \End_{\Q}(V)$, hence with the algebra generated by
$\su(B,\ba)$. 

Suppose that $mq\geq 3$. Then $\su(B,\ba)$ generates
$B$ as an algebra. So the generic endomorphism algebra $A$
is contained in the commutant
of $B$ in $\End_{\Q}(V)$, or equivalently in $\End_{F_0}(V)$,
which is $L$. Since $L$ is contained in $A$, we have shown
that the generic
endomorphism algebra $A$ is equal to $L$ in Type IV when
$\sum_{\nu=1}^g r_{\nu}s_{\nu}>0$ and $mq\geq 3$.

We now turn to the remaining cases, which involve Hodge structures
of low dimension over the totally real field $F$. For example,
case (d'):
let $L$ be of Type IV with $m=2$ and $q=1$,
while $r_{\nu}=s_{\nu}=1$ for $\nu=1,\ldots,g$. Thus $L$
is a CM field $F_0$, and $V$ has dimension 2
as an $L$-vector space.

The component $D^0$ of the Mumford-Tate domain
of $\Q$-Hodge structures on $(V,\langle,\rangle,L)$ given by
$r_{\nu}=s_{\nu}=1$ is isomorphic
to $\Gr(1,2)^g\cong (\C\P^1)^g$. I claim that every Hodge structure
in $D^0$ has extra endomorphisms; more precisely, the generic
endomorphism algebra for $\Q$-Hodge structures in $D^0$ is
of Type III, a totally definite quaternion algebra $L_2$ over
the totally real field $F$. Here $[L_2:\Q]=2[L:\Q]$ and so
$V$ has dimension 1 as an $L_2$-vector space. We showed above that
the Type III period domain for $(V,\langle,\rangle,L_2)$ is
$\Gr_{\Lag}(1,2)^g\cong (\C\P^1)^g$, and that the generic endomorphism
algebra for $\Q$-Hodge structures in that domain is equal to $L_2$.
So it suffices to show that there is an inclusion $L\subset L_2$
of algebras with involution, compatibly with the actions of $L$
and $L_2$ on $V$. (Then there is an obvious inclusion from the period
domain for $L_2$ into the one for $L$, which must be an isomorphism
since both domains are isomorphic to $(\C\P^1)^g$.)

For convenience, we work at first with the weaker assumption
that $m=2$, $q=1$, and $\sum_{\nu=1}^g r_{\nu}s_{\nu}>0$.

One way to find $L_2$ is that $SU(B,\ba)$ is a form
of $SL(2)$ over $F$, and so it is isomorphic to $SL(1,L_2)$ for a unique
quaternion algebra $L_2$ over $F$ \cite[Theorem 26.9]{KMRT}. By our earlier
discussion of Type IV,
the generic Hodge group $M$ for $D^0$ is normal in
the ``expected Hodge group'' $H=R_{F/\Q}U(B,\ba)$.
Since the period domain $D^0$ has dimension $g>0$, $M$ is not commutative
(as there are only countably many $\Q$-Hodge structures of CM type).
Since $H$ is the product of the $\Q$-simple group $R_{F/\Q}SU(B,\ba)
=R_{F/\Q}SL(1,L_2)$
with an abelian subgroup, $M$ must contain $R_{F/\Q}SL(1,L_2)$.
We have $M\subset H\subset SO(V_{\Q})$, where $SO(V_{\Q})$
is $\R$-anisotropic
(equivalently, its group of real points is compact) because
$\langle,\rangle$ is positive definite on $V$. So $M$ is also
$\R$-anisotropic, which means that the quaternion algebra $L_2$
over $F$ is totally definite, that is, of Type III.

Therefore, the generic endomorphism algebra $A$ for $D^0$
commutes with $R_{F/\Q}SL(1,L_2)$. So $A$ commutes with
the $\Q$-algebra generated by the Lie algebra of $R_{F/\Q}SL(1,L_2)$,
which is $\ker(\tr\colon L_2\arrow F)$. That algebra is equal to $L_2$.

The homomorphism $SL(1,L_2)\arrow SU(B,\ba)$ gives a homomorphism
$L_2\arrow B\arrow \End_{\Q}(V)$ of algebras with involution.
Since $[L_2:\Q]=2[L:\Q]$, $V$
has dimension 1 as an $L_2$-vector space. So the commutant
of $L_2$ in $\End_{\Q}(V)$ is isomorphic to $L_2^{\op}$.
So $L\subset A\subset L_2^{\op}$. Since $[L_2^{\op}:\Q]=2[L:\Q]$,
the generic endomorphism algebra
$A$ for $D^0$ must be either $L$ or $L_2^{\op}$.

We have an inclusion from the Mumford-Tate domain for $(V,\langle,
\rangle,L_2^{\op})$ into the one for $(V,\langle,\rangle,L)$.
The first is isomorphic to $(\C\P^1)^g$ and the second is isomorphic
to $D=(\coprod_{r_{\nu}=0}^2\Gr(r_{\nu},2))^g$, as shown earlier.
So the Mumford-Tate domain for $L_2^{\op}$ must be equal to the unique
component $D^0$ of $D$ of dimension $g$, the one with
$r_{\nu}=s_{\nu}=1$ for $\nu=1,\ldots,g$. We have thus shown that
every $\Q$-Hodge structure in that component $D^0$ has extra
endomorphisms, with the generic endomorphism algebra being
the Type III quaternion algebra $L_2^{\op}$ over $F$. At the same time,
the argument shows that when $\sum_{\nu} r_{\nu}s_{\nu}$ is greater
than 0 but less than $g$, the generic endomorphism algebra
for that component of $D$ is equal to the ``expected'' algebra $L$.

Next, consider case (e'):
let $L$ be of Type IV with $m=1$ and $q=2$,
while $r_{\nu}=s_{\nu}=1$ for $\nu=1,\ldots,g$. Thus $L$
is a quaternion algebra over a CM field $F_0$, and $V$ has dimension 1
as an $L$-vector space.

Since $r_{\nu}+s_{\nu}=mq$ for $\nu=1,\ldots,g$, we have $r_{\nu}
=s_{\nu}=1$ for all $\nu$. 
This component $D^0$ of the Mumford-Tate domain
of $\Q$-Hodge structures on $(V,\langle,\rangle,L)$ is isomorphic
to $\Gr(1,2)^g\cong (\C\P^1)^g$. I claim that every Hodge structure
in this domain has extra endomorphisms. More precisely, every
Hodge structure in $D^0$ is non-simple, with $V=W^{\oplus 2}$.
The generic endomorphism algebra for $W$ is
of Type III, a totally definite quaternion algebra which we call
$L_2^{\op}$ over
the totally real field $F$. Here $[L_2^{\op}:\Q]=[L:\Q]/2$ and so
$W$ has dimension 1 as an $L_2^{\op}$-vector space. We showed above that
the Type III period domain for $(W,\langle,\rangle,L_2^{\op})$ is
$\Gr_{\Lag}(1,2)^g\cong (\C\P^1)^g$, and that the generic endomorphism
algebra for $\Q$-Hodge structures in that domain is equal to $L_2^{\op}$.
So it suffices to show that there is an inclusion $L\subset M_2(L_2^{\op})$
of algebras with involution, compatibly with the actions of $L$
and $M_2(L_2^{\op})$ on $V$.
(Then there is an obvious inclusion from the period
domain for $L_2^{\op}$ into the one for $L$, which must be an isomorphism
since both domains are isomorphic to $(\C\P^1)^g$.)

For convenience, we work at first with the weaker assumption
that $m=1$, $q=2$, and $\sum_{\nu=1}^g r_{\nu}s_{\nu}>0$.

As in case (d'), the assumptions imply that the generic
Mumford-Tate group $M$ for the given component $D^0$
contains $R_{F/\Q}SU(B,\ba)$, and that
$SU(B,\ba)$ is isomorphic to $SL(1,L_2)$
for a unique totally definite quaternion algebra $L_2$ over
the totally real field $F$. As in case (d'), the generic
endomorphism algebra $A$ for $D^0$ commutes with $R_{F/\Q}
SL(1,L_2)$ and hence with the algebra $L_2\subset \End_{\Q}(V)$.

The homomorphism $SL(1,L_2)\arrow SU(B,\ba)$ gives a homomorphism
$L_2\arrow B\arrow \End_{\Q}(V)$ of algebras with involution.
Since $[L_2:\Q]=[L:\Q]/2$, $V$
has dimension 2 as an $L_2$-vector space. So the commutant
of $L_2$ in $\End_{\Q}(V)$ is isomorphic to $M_2(L_2^{\op})$.
So $L\subset A\subset M_2(L_2^{\op})$. Since $[M_2(L_2^{\op}):\Q]=2[L:\Q]$,
the generic endomorphism algebra
$A$ for $D^0$ must be either $L$ or $M_2(L_2^{\op})$.

We have an inclusion from the Mumford-Tate domain for $(W,\langle,
\rangle,L_2^{\op})$ into the one for $(V,\langle,\rangle,L)$, by sending
$W$ to $V:=W^{\oplus 2}$.
The first is isomorphic to $(\C\P^1)^g$ and the second is isomorphic
to $D=(\coprod_{r_{\nu}=0}^2\Gr(r_{\nu},2))^g$, as shown earlier.
So the Mumford-Tate domain for $L_2$ must be equal to the unique
component $D^0$ of $D$ of dimension $g$, the one with
$r_{\nu}=s_{\nu}=1$ for $\nu=1,\ldots,g$. We have thus shown that
in case (e'),
every $\Q$-Hodge structure in that component $D^0$ has extra
endomorphisms, with $V$ being a direct sum $V=W^{\oplus 2}$
of $\Q$-Hodge structures and $W$ generically having endomorphism algebra
equal to the Type III quaternion algebra $L_2^{\op}$ over $F$.
At the same time,
the argument shows that when $\sum_{\nu} r_{\nu}s_{\nu}$ is greater
than 0 but less than $g$, the generic endomorphism algebra
for that component of $D$ is equal to the ``expected'' algebra $L$.

Next, consider case (g'), with $L$ of Type I,
$m=4$, and $\disc(V)=1\in F^*/(F^*)^2$.
So $L$ is equal to $F$, a totally real field. The Mumford-Tate
domain $D$ for $(V,\langle,\rangle,L)$ is $\Gr_{\isot}(2,4)^g
\cong (\C\P^1\coprod \C\P^1)^g$. 

For each component $D^0$ of $D$, the generic Mumford-Tate group $M$
is normal in the ``expected'' Mumford-Tate group $H=R_{F/\Q}SO(_F V)$,
as shown earlier. Since $_F V$ has discriminant 1 in $F^*/(F^*)^2$,
$SO(_F V)$ is the product of two subgroups, $SL(1,L_2)SL(1,L_2^{\op})$,
where $L_2$ is a quaternion algebra over $F$ \cite[Corollary 15.12]{KMRT}.
We know that $M$ is nontrivial,
since a Hodge structure $V$ in $D^0$ has $V^{2,0}\neq 0$;
so the generic Mumford-Tate group $M$ for $D^0$
is either $R_{F/\Q}SL(1,L_2)$, $R_{F/\Q}SL(1,L_2^{\op})$, or all of $H$.
The generic endomorphism algebra for $D^0$ is the centralizer
of the generic Mumford-Tate group in $\End_{\Q}(V)$. So
the generic endomorphism algebra is $L_2^{\op}$, $L_2$,
or $F$, respectively. (Here we are thinking of $\End_F(V)$
as $L_2\otimes_F L_2^{\op}$; of course, $L_2^{\op}$ is isomorphic
to $L_2$, because $L_2$ is a quaternion algebra.)

Since $V$ has dimension 1 as an $L_2$-vector space,
the Mumford-Tate domain for $(V,\langle,\rangle,L_2)$ is isomorphic
to $(\C\P^1)^g$. The generic endomorphism algebra for that domain
of Type III is $L_2$. Likewise, the Mumford-Tate domain
for $(V,\langle,\rangle,L_2^{\op})$ is a different copy
of $(\C\P^1)^g$ inside $D$. We conclude that the generic endomorphism
algebra for these two connected components of $D$ is isomorphic
to $L_2$ ($\cong L_2^{\op})$, while the generic endomorphism
algebra for each of the other $2^g-2$ components of $D$ is the
``expected'' algebra $L=F$.

Next, consider case (b'), with $L$ of Type II, $m=2$,
and $\disc(B,\ba)=1\in F^*/(F^*)^2$. So $L$
is a totally indefinite quaternion algebra over the totally real
field $F$. 
Let $S_1,\ldots,S_g$ be the simple modules (of real dimension 2)
for the ring $L\otimes_{\Q}\R\cong (M_2(\R))^g$.
The Mumford-Tate domain $D$ for $(V,\langle,\rangle,L)$
is $\Gr_{\isot}(2,4)^g\cong (\C\P^1\coprod \C\P^1)^g$. 

For each component $D^0$ of $D$, the generic Mumford-Tate group $M$
is normal in the ``expected'' Mumford-Tate group
$H=R_{F/\Q}O^{+}(B,\ba)$,
as shown earlier. Since $(B,\ba)$ has discriminant 1 in $F^*/(F^*)^2$,
$O^{+}(B,\ba)$ is the product of two subgroups,
$SL(1,L_2)SL(1,L_3)$,
where $L_2$ and $L_3$ are quaternion algebras
over $F$ \cite[Corollary 15.12]{KMRT}.
We know that $M$ is nontrivial,
since a Hodge structure $V$ in $D^0$ has $V^{2,0}\neq 0$;
so the generic Mumford-Tate group $M$ for $D^0$
is either $R_{F/\Q}SL(1,L_2)$, $R_{F/\Q}SL(1,L_3)$, or all of $H$.
The generic endomorphism algebra for $D^0$ is the centralizer
of the generic Mumford-Tate group in $\End_{\Q}(V)$. So
the generic endomorphism algebra is $M_2(L_2^{\op})$, $M_2(L_3^{\op})$,
or $L$, respectively. (Here we are thinking of $B$
as $L_2\otimes_F L_3$, where $L_2$ and $L_3$ both have the canonical
symplectic involution. The whole algebra $\End_F(V)$
is the tensor product $L\otimes_F L_2\otimes_F L_3$, in this situation.)

Using the algebra $M_2(L_2^{\op})\subset \End_{\Q}(V)$, we can view
the $\Q$-vector space $V$ as a direct sum $V=W_2^{\oplus 2}$. 
Since $W_2$ has dimension 1 as an $L_2^{\op}$-vector space,
the Mumford-Tate domain for $(W_2,\langle,\rangle,L_2^{\op})$ is isomorphic
to $(\C\P^1)^g$. The generic endomorphism algebra for $W_2$ in that domain
of Type III is $L_2^{\op}$, as have shown. Likewise, the inclusion 
$M_2(L_3^{\op})\subset \End_{\Q}(V)$ gives a different decomposition
$V=W_3^{\oplus 2}$. We can view
the Mumford-Tate domain
for $(W_3,\langle,\rangle,L_3^{\op})$ as a different copy
of $(\C\P^1)^g$ inside $D$.
We conclude that in case (b'), the generic $\Q$-Hodge
structure in each of these two connected components of $D$ is non-simple,
of the form $W^{\oplus 2}$ where $W$ has endomorphism algebra
$L_2^{\op}$ or $L_3^{\op}$ of Type III, respectively. We also see that
the generic endomorphism
algebra for each of the other $2^g-2$ components of $D$ is the
``expected'' algebra $L$.

It remains to consider the cases where the generic $\Q$-Hodge
structure with endomorphisms by $L$ is in fact of CM type.
Every component $D^0$ of the Mumford-Tate domain contains a CM point
\cite[Lemma VI.C.1]{GGK},
and there are only countably many CM points in any Mumford-Tate domain.
So the generic Hodge structure in $D^0$ is of CM type
if and only if $D^0$ is a point. By the formula
for the dimension of $D^0$, these cases are:\\
(f') Type I, $m=2$.\\
(a') Type II, $m=1$.\\
(c') Type IV, $\sum_{\nu=1}^g r_{\nu}s_{\nu}=0.$

Let $(V,\langle,\rangle,L)$ be in case (f'). So $L$ is a totally real field
$F$ and $\dim_F(V)=2$. The ``expected'' Hodge group $H$ as defined earlier
is $R_{F/\Q}SO(_F V)$, which is commutative. Since the Hodge
group $M$ of $V$ is a normal subgroup of $H$, we see directly that
$M$ is commutative; that is, $V$ is of CM type. 

As an $F$-vector space of dimension 2, $V$ has a canonical symmetric
bilinear form $(,)$ (Lemma \ref{polar}). This form is positive definite,
and so its discriminant (the negative of the determinant, in this case)
is totally
negative in $F^*/(F^*)^2$. So $F_0:=F(\sqrt{\disc(_F V)})$ is a totally
imaginary quadratic extension of $F$.
The ``expected''
endomorphism algebra of $V$ is the centralizer of $R_{F/\Q}SO(_F V)$
in $\End_{\Q}(V)$, which is the CM field $F_0$. The actual endomorphism
algebra contains $F_0$. Since $\dim_{F_0}(V)=1$, we have constructed
enough endomorphisms to show again that $V$
is of CM type.

Next, let $(V,\langle,\rangle,L)$ be in case (a'): Type II and $m=1$.
So $L$ is a totally 
indefinite quaternion algebra over a totally real field $F$,
and $\dim_L(V)=1$. Let $B$ be the centralizer of $L$ in $\End_{\Q}(V)$,
or equivalently in $\End_F(V)$.
The ``expected'' Hodge group $H$ as defined earlier is
$R_{F/\Q}O^+(B,\ba)$, which is commutative since
$O^+(B,\ba)$ is an $F$-form of $SO(2m)=SO(2)$. So we see directly
that $V$ is of CM type.

The discriminant $\disc(B,\ba)$ in $F^*/(F^*)^2$ is totally
negative, using the positivity of the symmetric bilinear form
$(,)$ on the 4-dimensional $F$-vector space
$V$. So $F_0:=F(\sqrt{\disc(B,\ba)})$ is a totally
imaginary quadratic extension field of $F$. The ``expected''
endomorphism algebra as defined earlier is the centralizer
of $R_{F/\Q}O^+(B,\ba)$, which is the matrix algebra $M_2(F_0)$.
The actual endomorphism algebra of the $\Q$-Hodge structure
$V$ contains $M_2(F_0)$. We conclude that the $\Q$-Hodge structure $V$
is isomorphic
to a direct sum $V=W^{\oplus 2}$, where $W$ has endomorphisms
by the CM field $F_0$, and $\dim_{F_0}W=1$. We have
constructed enough endomorphisms to see again that $V$
is of CM type.

Finally, let $(V,\langle,\rangle,L)$ be in case (c'): Type IV
with $\sum_{\nu=1}^g r_{\nu}s_{\nu}=0$. So $L$ is a central simple
algebra with unitary involution over a CM field $F_0$.
We write $2n=m[L:\Q]$ and
$[L:F_0]=q^2$. Let $B\cong M_m(L^{\op})$ be the centralizer
of $L$ in $\End_{\Q}(V)$, or equivalently in $\End_{F_0}(V)$.
The ``expected Hodge group'' $H$ as defined earlier
is $H=R_{F/\Q}U(B,\ba)$; here $U(B,\ba)$
is an $F$-form of $GL(mq)$. The Hodge group $M$ of $V$ is a normal
connected $\Q$-subgroup of $H$, and we also know that $M$ is commutative
because the given component $D^0$ of the Mumford-Tate domain
has dimension zero. So $M$ is
a subgroup of the center of $H$, which is the torus $R_{F/\Q}T$,
where $T$ is the 1-dimensional
torus $\ker(N\colon R_{F_0/F}G_m\arrow G_m)$ over $F$.

So the endomorphism algebra of the $\Q$-Hodge structure
$V$ contains the centralizer
of $R_{F/\Q}T$ in $\End_{\Q}(V)$, which is the matrix
algebra $\End_{F_0}(V)\cong M_{mq^2}(F_0)$. Therefore, the $\Q$-Hodge
structure $V$ is a direct sum
$V=W^{\oplus mq^2}$ for a $\Q$-Hodge structure $W$ with endomorphisms
by $F_0$ such that $\dim_{F_0}(W)=1$. Thus we have constructed
enough endomorphisms to see again that $V$ is of CM type.
\end{proof}

\begin{remark}
As an addendum to Theorem \ref{main}, we can say when a CM Hodge
structure has more than the expected endomorphism algebra. For
CM abelian varieties, this was worked out by Shimura
\cite[Proposition 26]{Shimurabook}. Namely, let $X$ be a complex abelian
variety of dimension $g$ with a homomorphism $F_0\arrow \End(X)_{\Q}$
for a CM field $F_0$
of degree $2g$ over $\Q$. The isogeny type of $X$ is described by a
{\it CM type }on $F_0$,
meaning a set $\Phi$ of $g$ complex embeddings of $F_0$
such that every complex embedding is in $\Phi\cup \overline{\Phi}$.
Then $\End(X)_{\Q}$ is strictly larger than $F_0$ if and only if there
is a strictly smaller CM subfield $K_0$ of $F_0$, with subfield $K$
fixed by complex conjugation, such that any two elements of $\Sigma$
which agreee on $K$ also agree on $K_0$. When this happens, $X$
is isogenous to a power of the CM abelian variety with endomorphisms
by $K_0$ and CM type $\Sigma|_{K_0}$.

Essentially the same statement holds for CM Hodge structures of type
$(g,0,\ldots,0,g)$ of any weight $w\geq 1$. Namely, the CM Hodge
structures of type $(g,0,\ldots,0,g)$ with a homomorphism
$F_0\arrow \End_{\QHS}(V)$ for a CM field $F_0$ of degree $2g$
are classified by CM types on $F_0$, just as in weight 1.
(In particular, these Hodge structures are all polarizable.)
It follows that the equivalence of categories of Hodge structures
of type $(g,g)$ to Hodge structures of type $(g,0,\ldots,0,g)$
given by renaming $V^{1,0}\subset V\otimes_{\Q}\C$
as $V^{w,0}$ restricts to an equivalence from the CM Hodge structures
of weight 1 to those of weight $w$. Therefore, we have the same
criterion as in the previous paragraph for when a CM Hodge structure
of type $(g,0,\ldots,0,g)$ has more than the expected endomorphism
algebra.
\end{remark}

\section{Hodge structures not generated by curves}

In this section, we show that the $\Q$-Hodge structures considered
in this paper, those of weight at least 2 with Hodge numbers
$(n,0,\ldots,0,n)$, are not in the subcategory
generated by curves (or equivalently by abelian varieties),
except when they are of CM type.

Define the subcategory of $\Q$-Hodge structures {\it generated by curves}
to be the subcategory of Hodge structures generated
by $H^1$ of smooth complex projective curves together with the Hodge
structure $\Q(j)$ for integers $j$ by taking direct sums,
tensor products, and direct summands. This can also be described
as the subcategory generated by abelian varieties. Every
Hodge structure of CM type belongs to the subcategory generated by
curves. The Kuga-Satake construction shows that every polarizable
$\Q$-Hodge structure of type $(1,b,1)$ is in the subcategory
generated by curves \cite{Voisinkuga}.

We use the following result of Deligne's \cite[Lemma 5]{Schnell}.

\begin{theorem}
\label{deligne}
Let $V$ be a $\Q$-Hodge structure which belongs to the subcategory
generated by curves.
Then the Hodge structure on the Lie algebra of the Mumford-Tate
group of $V$ is of type $\{(-1,1),(0,0),(1,-1)\}$.
\end{theorem}

\begin{corollary}
\label{notcurve}
Let $V$ be a $\Q$-Hodge structure of weight $w\geq 2$
with Hodge numbers $(n,0,\ldots,0,n))$.
If $V$ is not of CM type, then it is not in the subcategory
generated by curves.
\end{corollary}

\begin{proof}
The Hodge structure on $\End_{\Q}(V)$ is of type $\{(-n,n),(0,0),
(n,-n)\}$. The Lie algebra $\mt$ of the Mumford-Tate group $\MT(V)$
is a sub-Hodge structure of $\End_{\Q}(V)$. If 
the Hodge structure $V$ is in the subcategory generated by curves,
then in particular it is polarizable. Also,
by Theorem \ref{deligne},
$\mt$ is of type $\{(-1,1),(0,0),(1,-1)\}$, and so $\mt$ must
be of type $\{(0,0)\}$. Equivalently, the homomorphism
$R_{\C/\R}G_m\arrow \MT(V)_{\R}\arrow GL(\mt)$ that describes
the Hodge structure on $\mt$ is trivial.
Since $R_{\C/\R}G_m$ is Zariski dense in the Mumford-Tate group
as a $\Q$-group, it follows that the conjugation
homomorphism $\MT(V)\arrow GL(\mt)$ is trivial.
That is, the connected $\Q$-group $\MT(V)$ is commutative.
So $V$ is of CM type.
\end{proof}

\small \sc UCLA Mathematics Department, Box 951555,
Los Angeles, CA 90095-1555

totaro@math.ucla.edu
\end{document}